\documentclass[a4paper,11pt]{article}
\usepackage[ngerman, english]{babel}
\usepackage[T1]{fontenc}
\usepackage[utf8]{inputenc}
\usepackage{amsmath}
\usepackage{amssymb}
\usepackage{amsthm}
\usepackage{graphicx}
\usepackage{enumerate}
\usepackage{lmodern}
\usepackage{fancyvrb}
\usepackage[plainpages=false]{hyperref}
\usepackage{caption}
\newtheorem{defi}{Definition}
\newtheorem{theo}[defi]{Theorem}
\newtheorem{lem}[defi]{Lemma}
\theoremstyle{remark}

\title{Even factors in edge-chromatic-critical graphs with a small number of divalent vertices}

\author{Eckhard Steffen, Isaak H.~Wolf \thanks{funded by the Deutsche Forschungsgemeinschaft (DFG German Research Foundation) - STE 792/3-1} \\
		Institute for Mathematics, Paderborn University, \\ Warburger Str. 100, 33098 Paderborn,
		Germany. \\ es@upb.de, isaak.wolf@upb.de}

\date{}

\begin{document}

\maketitle

\begin{abstract}
A finite simple connected graph $G$ with maximum degree $k$ is $k$-critical if it has chromatic index $\chi'(G)=k+1$ and $\chi'(G-e)=k$ for every edge $e\in E(G)$. Bej and the first author \cite{BS} raised the question whether every $k$-critical graph has an even factor. We prove that
every $k$-critical graphs with at most $2k-6$ vertices of degree 2 has an even factor.
\end{abstract}

We consider finite simple graphs with vertex set $V(G)$ and edge set $E(G)$. The degree of a vertex $v \in V(G)$ is denoted by $d(v)$ and $\Delta(G)$ 
denotes the maximum degree of a vertex of $G$.
A vertex of degree 2 is also called a \textit{divalent} vertex.
A graph is \textit{even} if each component is eulerian and different from $K_1$.

For a subset $A$ of $V(G)$, the subgraph induced by $V(G)-A$ is denoted by $G-A$. If no two vertices in $A$ are connected by an edge, then $A$ is \textit{stable}. For a subset $E'$ of $E(G)$, the graph obtained from $G$ by deleting the edges in $E'$ is denoted by $G-E'$ (if $E'=\{e\}$, then we use the notation $G-e$). 

A \textit{$k$-edge-colouring} of a graph $G$ is a function $\varphi: E(G) \to \{1,...,k\}$; the elements of $\{1,...,k\}$ are called \textit{colours}. A $k$-edge-colouring $\varphi$ is \textit{proper}, if no two adjacent edges receive the same colour. For a vertex $v \in V(G)$ and a colour $i \in \{1,..,k\}$, we say colour $i$ is \textit{present} at $v$ if an edge incident to $v$ is coloured with colour $i$. Otherwise, colour $i$ is \textit{missing} at $v$. The set of colours present at $v$ is denoted by $\varphi(v)$; the set of colours missing at $v$ is denoted by $\bar{\varphi}(v)$.
If a proper $k$-edge-colouring of $G$ exists, then $G$ is \textit{$k$-edge-colourable}. The \textit{chromatic index}, denoted $\chi'(G)$, is the smallest integer $k$ such that $G$ is $k$-edge-colourable. Vizing \cite{V_Result}
proved the fundamental result on edge-colouring simple graphs by showing that 
the chromatic index of a graph $G$ is either $\Delta(G)$ or $\Delta(G)+1$. 

An edge $e \in E(G)$ is \textit{critical}, if $\Delta(G)=k$, $\chi'(G)=k+1$ and $\chi'(G-e)=k$. If $G$ is connected and all edges of $G$ are critical, then $G$ is \textit{$k$-critical}. Clearly,
every graph $H$ with $\chi'(H) = \Delta(H)+1$ contains
a $\Delta(H)$-critical subgraph. There had been several conjectures
with regard to the order or to (near) perfect matchings of critical graphs,
which all turned out to be false, see \cite{BS} for a survey. 

The situation changes when we consider 2-factors. 
In 1965, Vizing \cite{V} conjectured that every critical graph has a 2-factor. This conjecture has been verified for some specific classes of critical graphs as overfull graphs \cite{GS} or critical graphs
with large maximum degree in relation to their order \cite{Chen_Shan_2017, Luo_Zhao_2013}. Furthermore, some equivalent formulations or reduction to
some classes of critical graphs as e.g. critical graphs of even order are proved in \cite{BS, CJL_2020}. 
All these approaches have not yet led to significant progress in answering the question whether critical graphs have a 2-factor. 

To gain more insight into structural properties of critical graphs 
it might be useful to investigate slightly easier statements about factors in critical graphs. In \cite{BS} Bej and the first author conjectured that (1) every critical graph has a path-factor and (2) every critical graph has an even factor.  

The first conjecture is proved in \cite{KS}. For the second conjecture 
note that every bridgeless graph with minimum degree at least $3$ has an even factor \cite{F}. Thus, conjecture (2) is reduced to critical graphs with divalent vertices. Our main result is the following:

\begin{theo} \label{main result}
	Let $k \geq 3$ and $G$ be a $k$-critical graph. If 
	$G$ has at most $2k-6$ divalent vertices, then $G$ has an even factor.
\end{theo}

\section{Preliminaries and useful lemmas}

The set of edges with one end in $A_{1}$ and the other in $A_{2}$ is denoted by $E_{G}(A_{1},A_{2})$, where $A_{i}$ is a vertex set, a single vertex (in this case, one end in $A_{i}$ means one end in $\{A_{i}\}$) or a subgraph (in this case, one end in $A_{i}$ means one end in $V(A_{i})$). Furthermore, $e_{G}(A_{1},A_{2})$ is the cardinality of $E_{G}(A_{1},A_{2})$. The set of neighbours of a vertex $v \in V(G)$ is denoted by $N(v)$; for a subset $A$ of $V(G)$ the \textit{neighbourhood} of $A$, denoted $N(A)$, is $\bigcup_{v \in A} N(v) -A$. If $G'$ is a subgraph of $G$, then we write $N(G')$ instead of $N(V(G'))$.

 Let $\varphi$ be a proper $k$-edge-colouring of a graph $G$ and $v \in V(G)$. For two 
 different colours $i,j \in \{1,...,k\}$, the subgraph induced by the edges that are coloured $i$ or $j$ is denoted by $K(i,j)$. Its components are called \textit{$(i,j)$-Kempe chains} or sometimes just \textit{Kempe chains}. 
 Clearly, a Kempe chain is a path or a circuit. If $\{i,j\} \cap \varphi(v) \neq \emptyset$, then the unique component of $K(i,j)$ that contains $v$ is denoted by $P_{v}^{\varphi}(i,j)$. 
 We will omit the upper index if this does not cause any ambiguity. 
 A new proper $k$-edge-colouring, denoted by $\varphi/P_{v}(i,j)$, can be obtained from $\varphi$ by interchanging 
 colours $i$ and $j$ in $P_{v}(i,j)$.

 In the proofs of Lemma~\ref{cutlemma} and \ref{3-cutlemma} we will use the following basic observation without reference:
Let $G$ be a graph with a critical edge $vw$ and let $\varphi$ be a proper $\Delta(G)$-edge-colouring of $G-vw$. If colour $i$ is missing at $v$ and $j$ is missing at $w$, then colour $i$ is present at $w$, colour $j$ is present at $v$ and $P_{v}^{\varphi}(i,j)$ is a $v,w$-path.

\begin{lem} \label{cutlemma}
Let $G$ be a graph with $\Delta(G)=k$, $\chi'(G)=k+1$, and let $A \subseteq V(G)$ 
be a set of vertices such that 
\begin{itemize}
\item[$\bullet$] $e_{G}(A,v)=1$ for every $v \in N(A)$, and
\item[$\bullet$] $N(A)=\{x,y,w_{1},...,w_{l}\}$ with $l\geq1$, $d(y)\leq d(x)<k$ and $d(w_{i})=2$ for every $i \in \{1,...,l\}$.
\end{itemize}

If at least one edge in $E_{G}(A,\{w_{1},...,w_{l}\})$ is critical,
then $l > k(k-d(y))-d(x)+1$.
\end{lem}

\begin{proof}
Let $w \in\{w_{1},...,w_{l}\}$ be a divalent vertex, let $w'$ be the unique neighbour of $w$ that belongs to $A$, and let the edge $w'w$ be critical. 
\vspace{.2cm}

\underline{Claim 1:} There is a proper $k$-edge-colouring $\varphi$ of $G-w'w$ such that $\bar{\varphi}(w')=\varphi(w)=\{1\}$ and $1 \in \bar{\varphi}(x)$.

{\it Proof.} Since $w'w$ is critical there is a proper $k$-edge-colouring $\varphi'$ of $G-w'w$. Furthermore, $\bar{\varphi}'(w')=\varphi'(w)=\{i\}$ for a colour $i \in \{1,...,k\}$. Since $d(x)<k$, there is a colour $j$ that is missing at $x$. If $i=j \neq 1$, then we obtain a colouring with the desired properties by interchanging colours $i$ and $1$. If $i \neq j$, then $P_{w'}(i,j)$ is a $w',w$-path and thus does not contain $x$. Therefore, the colouring $\varphi''$, defined by $\varphi''=\varphi'/P_{w'}(i,j)$, satisfies $\bar{\varphi}''(w')=\varphi''(w)=\{j\}$ and $j \in \bar{\varphi}''(x)$. Again, if $j \not= 1$, then a colouring with the desired properties can be obtained by interchanging colours $j$ and $1$. Thus, Claim 1 is proved. 
\vspace{.2cm}

Now fix a proper $k$-edge-colouring $\varphi$ of $G-w'w$ with the properties stated in Claim 1. Define a set $M$ as follows:
\begin{align*}
M=\{(h,z,h'): \text{ }& z \in N(A)-\{w\}, \{h,h'\} \subseteq \varphi(z), h \neq h',\\& \varphi(e)=h \text{, where } e \text{ is the unique edge in } E_{G}(A,z)\}.
\end{align*}

We prove a lower bound for the number of triples in $M$, which will be used to obtain the lower bound for $l$.\\
For each triple $(h,z,h')$ of $M$ there is a unique Kempe chain $P$, that contains the two edges incident with $z$ that are coloured $h$ and $h'$. In this case we say \textit{$P$ contains $(h,z,h')$}. Furthermore, if $P$ is a path and $v$ is an end vertex of $P$, then we can interpret $P$ as a vertex-list starting with $v$. This gives an order of the vertices of $P$ and thus an order of the triples contained in $P$. We define the \textit{first} and the \textit{last} triple contained in $P$ (starting with $v$) in the natural way. An example is given in Figure~\ref{fig:1}.

\begin{figure}
	\centering
	\includegraphics[scale=0.37]{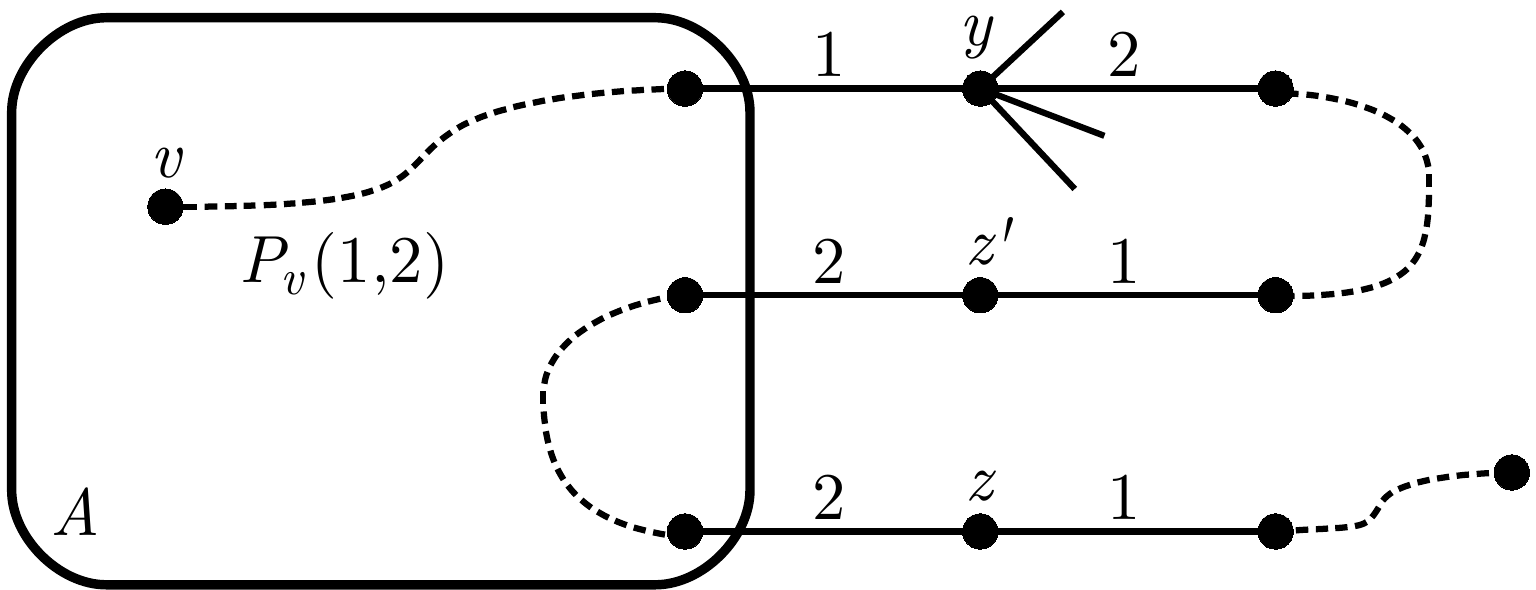}
	\caption{$P_{v}(1,2)$ contains $(1,y,2)$, $(2,z',1)$ and $(2,z,1)$. $(1,y,2)$ is the first and $(2,z,1)$ the last triple contained in $P_{v}(1,2)$ (starting with $v$).}
	\label{fig:1}
\end{figure}

If $i \in \{2,...,k\}$, then the Kempe chain $P_{w'}(1,i)$ is a path with end vertices $w'$ and $w$ and thus, it contains at least one triple of $M$. Therefore, we can define a subset $M_1$ of $M$ as follows:
For every $i \in \{2,...k\}$ let $(i_{1},z_{i},i_{2})$ be the last triple contained in $P_{w'}(1,i)$ (starting with $w'$). Let $M_1=\{(i_{1},z_{i},i_{2}):i \in \{2,...k\}\}$.
Figure~\ref{fig:2} shows an example. We note, that $\{i_1,i_2\}=\{1,i\}$ for every $i \in \{2,...,k\}$ and in particular, $x$ is not in a triple of $M_1$, since colour $1$ is missing at $x$.

\begin{figure}
	\centering
	\includegraphics[scale=0.37]{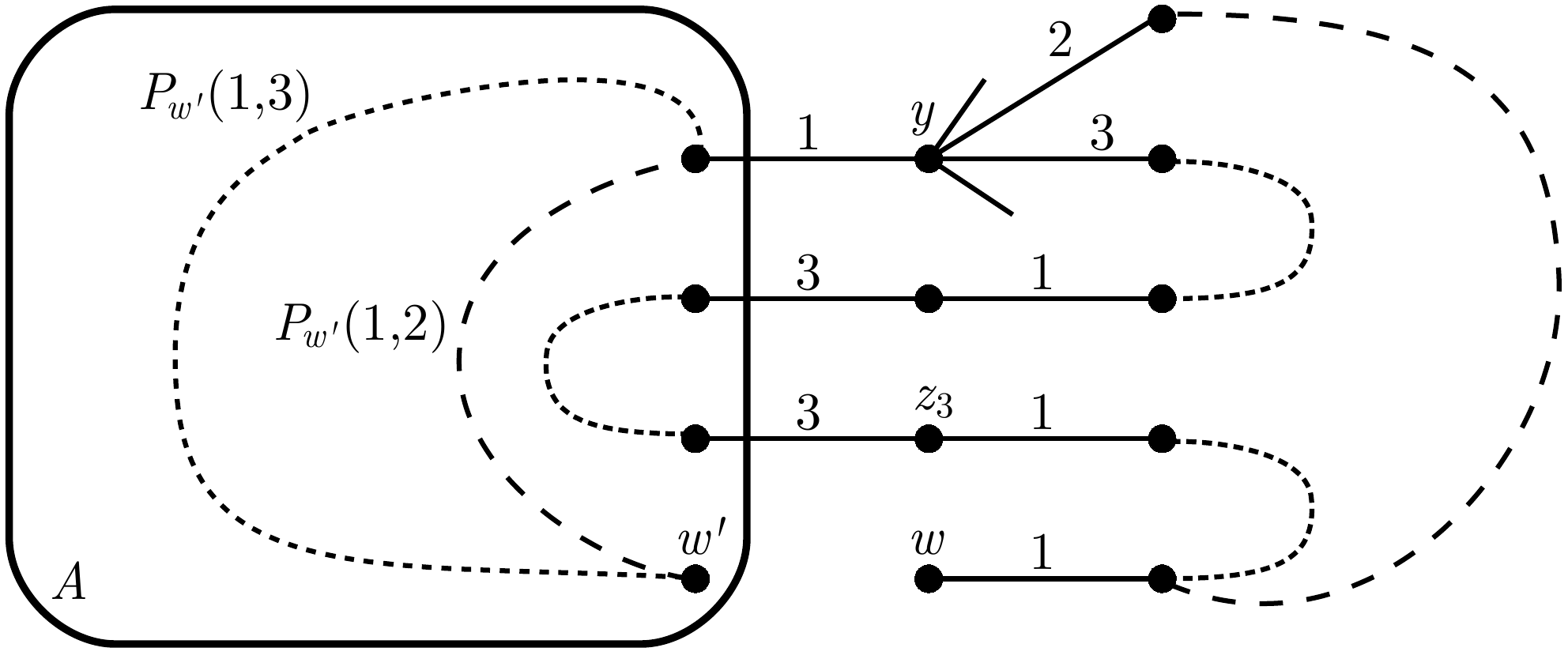}
	\caption{The triple $(3,z_3,1)$ is the last triple contained in $P_{w'}(1,3)$; $(1,y,2)$ is the last triple contained in $P_{w'}(1,2)$(starting with $w'$). Thus, $(3,z_3,1), (1,z_2,2) \in M_1$ where $z_2=y$.}
	\label{fig:2}
\end{figure}

\vspace{.2cm}
\underline{Claim 2:} $\vert M_{1} \vert = k-1$.

{\it Proof.}
Let $i, i'$ be two different colours of $\{2,...,k\}$. Then, $\{i_1,i_2\}=\{1,i\}\neq\{1,i'\}=\{i'_1,i'_2\}$, and hence $(i_{1},z_{i},i_{2}) \neq (i'_{1},z_{i'},i'_{2})$.

\vspace{.2cm}
\underline{Claim 3:} Let $i \in \{2,...,k\}$ and $j \in \bar{\varphi}(z_{i})$. Then $P_{z_{i}}(i_{1},j)$ contains a triple of $M$.

{\it Proof.}
Suppose, $P_{z_{i}}(i_{1},j)$ does not contain a triple of $M$. Then, the colouring $\varphi'$, defined by $\varphi'=\varphi/P_{z_{i}}(i_{1},j)$, satisfies $\bar{\varphi}'(w')=\varphi'(w)=\{1\}$. Since $(i_{1},z_{i},i_{2})$ is the last triple contained in  $P_{w'}(1,i)$ (starting with $w'$), the Kempe chain $P_{w}^{\varphi'}(i_{1},i_{2})$ has end vertices $w$ and $z_{i}$. In particular $P_{w}^{\varphi'}(i_{1},i_{2})$ is not a $w',w$-path. We have either $i_{1}=1$ or $i_{2}=1$, a contradiction. See Figure~\ref{fig:3} for an example. 

\begin{figure}
	\centering
	\includegraphics[scale=0.37]{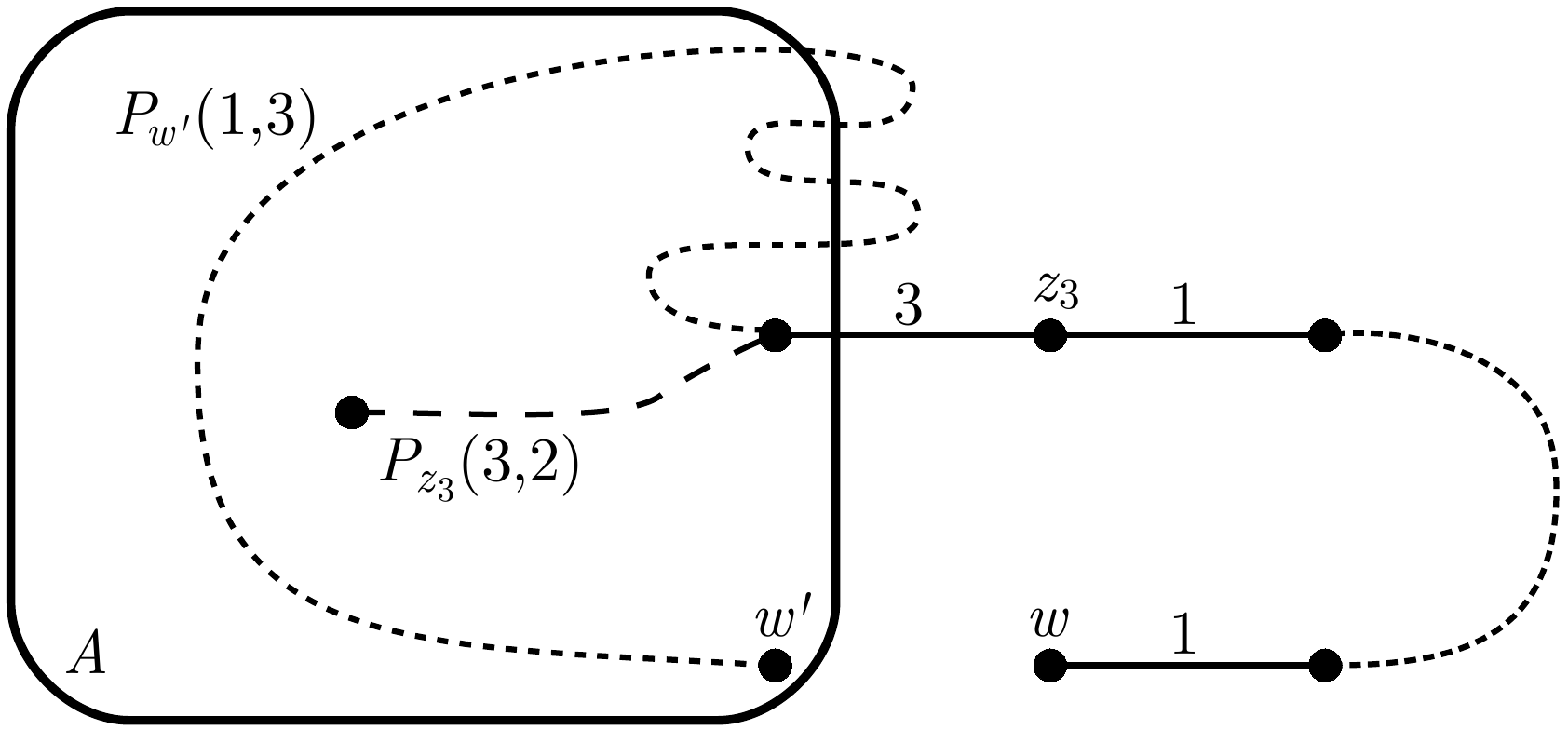}
	\caption{The triple $(3,z_3,1)$ is in $M_1$. Colour $2$ is missing at $z_3$. The Kempe chain $P_{z_3}(3,2)$ does not contain a triple of $M$. Interchanging colours $3$ and $2$ in $P_{z_3}(3,2)$ produces a contradiction, since $P_{w}(1,3)$ is not longer a $w',w$-path.}
	\label{fig:3}
\end{figure}

\vspace{.2cm}
Let 
	$M_{2}=\{(h,z,h'): \text{ } i \in \{2,...,k\}, j \in \bar{\varphi}(z_{i}),(h,z,h') \text{ is the first triple contained in } P_{z_{i}}(i_{1},j) \\ \text{ (starting with } z_{i})\}.$
An example is given in Figure~\ref{fig:4}.

\begin{figure}
\centering
\includegraphics[scale=0.37]{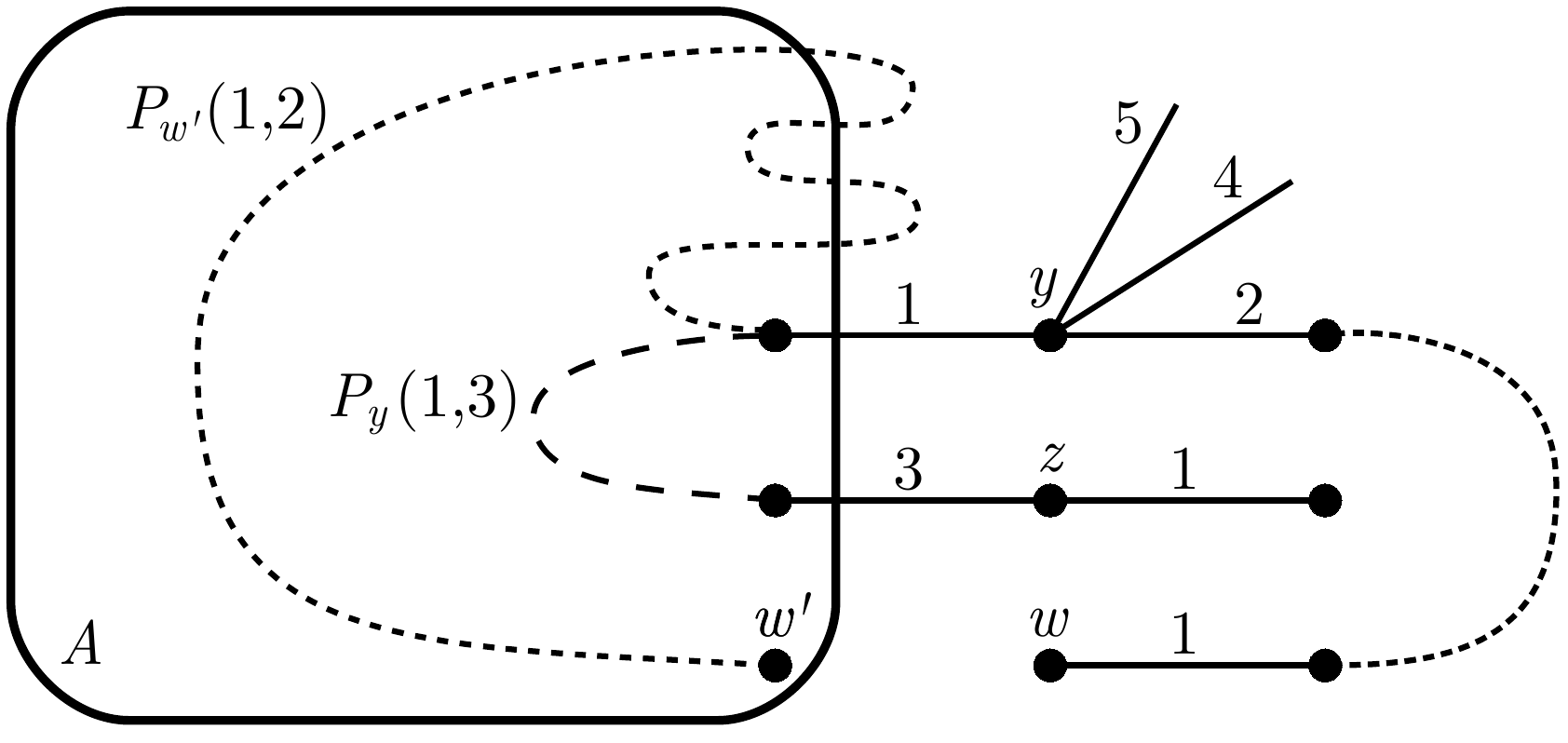}
\caption{The triple $(1,z_2,2)$ is in $M_1$, where $z_2=y$. Colour $3$ is missing at $y$. The triple $(3,z,1)$ is the first triple contained in $P_{y}(1,3)$ (starting with $y$) and thus in $M_2$.}
\label{fig:4}
\end{figure}

\vspace{.2cm}
\underline{Claim 4:} $M_{1} \cap M_{2}=\emptyset$.

{\it Proof.}
We have $z_{i} \notin \{w',w\}$ for every $i \in \{2,...,k\}$. Hence, every triple of $M_2$ is contained in a path with an end vertex that is neither $w'$ nor $w$, whereas every triple of $M_1$ is contained in a $w',w$-path.

\vspace{.2cm}
\underline{Claim 5:} $\vert M_{2} \vert = \vert \{(j,z_{i}): i \in \{2,...,k\}, j \in \bar{\varphi}(z_{i})\} \vert$.

{\it Proof.}
Let $i,i' \in \{2,...,k\}$, $j \in \bar{\varphi}(z_{i})$ and $j' \in \bar{\varphi}(z_{i'})$ such that $(j,z_i)\neq(j',z_{i'})$. Then, the paths $P_{z_{i}}(i_{1},j)$ and $P_{z_{i'}}(i'_{1},j')$ have at least one different colour or a different starting vertex (interpreted as a vertex-list starting with $z_{i}$ or $z_{i'}$ respectively). Therefore, these two paths have different first triples (in the case $z_{i} \neq z_{i'}$ we use the fact that both triples are first triples).

\vspace{.2cm}
\underline{Claim 6:} $\vert \{(j,z_{i}): i \in \{2,...,k\}, j \in \bar{\varphi}(z_{i})\} \vert \geq (k-d(y))(k-1)$.

{\it Proof.}
Since colour 1 is missing at $x$, there is no $i \in \{2,...,k\}$ with $z_{i}=x$. If  $z_i=z_{i'}$ for two different integers $i,i'$ of $\{2,...,k\}$, then $z_i=z_{i'}=y$, since every vertex in $N(A)\setminus \{x,y\}$ is divalent. Furthermore, the number of indices $i \in \{2,...,k\}$ with $z_{i}=y$ is at most $d(y)-1$. Vertex $y$ misses $k-d(y)$ colours whereas all other vertices in $N(A)-\{w,x\}$ miss $k-2$ colours. In conclusion:
\begin{align*}
&\vert \{(j,z_{i}): i \in \{2,...,k\}, j \in \bar{\varphi}(z_{i})\} \vert\\
\geq &(k-d(y))+(k-1-(d(y)-1))(k-2)\\
=&(k-d(y))(k-1).
\end{align*}

\vspace{.2cm}
We now prove that $l \geq k(k-d(y))-d(x)+1$.
Since $e_{G}(A,v)=1$ for every $v \in N(A)$ and all vertices in $N(A)-\{x,y\}$ are divalent, the inequality $l >\vert M \vert - (d(x)-1)-(d(y)-1)$ holds. By Claims 2, 4, 5 and 6, we have:
\begin{align*}
l &>\vert M \vert - (d(x)-1)-(d(y)-1)\\
&\geq \vert M_{1} \vert +\vert M_{2} \vert  - (d(x)-1)-(d(y)-1)\\
&\geq k-1+(k-d(y))(k-1)-(d(x)-1)-(d(y)-1)\\
&= k(k-d(y))-d(x)+1.
\end{align*} \end{proof}

\begin{lem} \label{3-cutlemma} Let $k > 3$ and 
let $G$ be a graph with $\Delta(G)=k$ and $\chi'(G)=k+1$. If $E'\subseteq E(G)$ is an inclusion-minimal edge-cut consisting of three critical edges, then no edge in $E'$ is incident to a divalent vertex.
\end{lem}

\begin{proof}
Let $G$ be a graph with $\Delta(G)=k$ and $\chi'(G)=k+1$. Furthermore, let $E'\subseteq E(G)$ be an inclusion-minimal edge-cut consisting of three critical edges $e_{1},e_{2}$ and $e_{3}$; let $A$ and $B$ be the components of $G-E'$, and let $e_{i}=x_{i}y_{i}$, where $x_{i}$ belongs to $A$ and $y_{i}$ to $B$. Let $G_A$ be the subgraph induced by $V(A) \cup \{y_{1},y_{2}, y_{3}\}$, and let $G_B$ be the subgraph induced by $V(B) \cup \{x_{1},x_{2}, x_{3}\}$.
We say a $k$-edge-colouring $\varphi$ of $G_A$ or $G_B$ is of
\begin{itemize}
\item type 1, if $\varphi(e_{1})=\varphi(e_{2})=\varphi(e_{3})$,
\item type 2, if $\varphi(e_{1})=\varphi(e_{2})\neq\varphi(e_{3})$,
\item type 3, if $\varphi(e_{1})=\varphi(e_{3})\neq\varphi(e_{2})$,
\item type 4, if $\varphi(e_{2})=\varphi(e_{3})\neq\varphi(e_{1})$,
\item type 5, if $\varphi(e_{1})\neq \varphi(e_{2})$, $\varphi(e_{1})\neq \varphi(e_{3})$, $\varphi(e_{2})\neq \varphi(e_{3})$.
\end{itemize}
Suppose to the contrary that there is an edge of $E'$ that is incident to a divalent vertex. We will show that there is a proper $k$-edge-colouring $\varphi_{A}$ of $G_A$ and a proper $k$-edge-colouring $\varphi_{B}$ of $G_B$ such that $\varphi_{A}$ and $\varphi_{B}$ can be combined to a proper $k$-edge-colouring of $G$ (by possibly relabeling the colours in one of the colourings), a contradiction.

In order to label the appearing colourings properly, we use the following definition: For an edge $e \in E'$, a $k$-edge-colouring $\varphi$ of $G-e$ and a colour $i \in \{1,...,k\}$, let $\varphi_{i}$ denote the $k$-edge-colouring of $G$ obtained from $\varphi$ by colouring $e$ with $i$.

Suppose to the contrary that $d(y_{1})=2$.
First of all we use the fact that $e_{1}$ is critical. Let $\varphi$ be a proper $k$-edge-colouring of $G-e_{1}$ such that w.l.o.g. $\bar{\varphi}(x_{1})=\varphi(y_{1})=\{1\}$ holds. Thus, for every $i \in \{2,...,k\}$ the Kempe chain $P_{x_{1}}(1,i)$ is an $x_{1},y_{1}$-path. Since $k>3$, at least two of these paths, say $P_{x_{1}}(1,2)$ and $P_{x_{1}}(1,3)$, contain w.l.o.g. $e_{2}$, which means $\varphi(e_{2})=1$. We first prove that the colouring $\varphi$ can be used to obtain a proper type 1 and a proper type 2 $k$-edge-colouring of $G_A$ and a proper type 3, a proper type 4 and a proper type 5 $k$-edge-colouring of $G_B$, no matter which colour the edge $e_{3}$ has received.

\underline{Case 1:  $\varphi(e_{3})=1$}\\
In this case, the colouring $\varphi_{1}\mid_{E(G_A)}$ is a proper type 1 $k$-edge-colouring of $G_A$. On the other hand, $\varphi_{2}\mid_{E(G_B)}$ is a proper type 4 $k$-edge-colouring of $G_B$. Furthermore, the colouring $\varphi'$, defined by $\varphi'=\varphi/P_{x_{3}}(1,2)$, satisfies $\varphi'(e_{3})=2$, while $\bar{\varphi}'(x_{1})=\varphi'(y_{1})=\{1\}$ and $\varphi'(e_{2})=1$ still hold. Therefore, $\varphi'_{1}\mid_{E(G_A)}$ is a proper type 2 $k$-edge-colouring of $G_A$, the colouring $\varphi'_{2}\mid_{E(G_B)}$ is a proper type 3 $k$-edge-colouring of $G_B$, and $\varphi'_{3}\mid_{E(G_B)}$ is a proper type 5 $k$-edge-colouring of $G_B$.

\underline{Case 2:  $\varphi(e_{3}) \neq 1$}\\
If $e_{3} \in P_{x_{1}}(1,\varphi(e_{3}))$, then the colouring $\varphi'$, defined by $\varphi'=\varphi/P_{x_{1}}(1,\varphi(e_{3}))$, satisfies $\bar{\varphi}'(x_{1})=\varphi'(y_{1})=\{\varphi(e_{3})\}$ and $\varphi'(e_{2})=\varphi'(e_{3})=1$. Hence, for any $i \in \{2,...,k\}\setminus\{\varphi(e_{3})\}$ the Kempe chain $P_{x_{1}}^{\varphi'}(i,\varphi(e_{3}))$ is not an $x_{1},y_{1}$-path, a contradiction. Therefore, we may assume $e_{3} \notin P_{x_{1}}(1,\varphi(e_{3}))$, which implies $e_{2} \in P_{x_{1}}(1,\varphi(e_{3}))$. As a consequence, the colouring $\varphi'$, defined by $\varphi'=\varphi/P_{x_{3}}(1,\varphi(e_{3}))$, satisfies $\bar{\varphi}'(x_{1})=\varphi'(y_{1})=\{1\}$ and $\varphi'(e_{2})=\varphi'(e_{3})=1$. Since $P_{x_{1}}^{\varphi'}(1,2)$ and $P_{x_{1}}^{\varphi'}(1,3)$ still contain $e_{2}$, Case 1 applies.

In both cases there is a proper type 1 and a proper type 2 $k$-edge-colouring of $G_A$, and a proper type 3, a proper type 4 and a proper type 5 $k$-edge-colouring of $G_B$.\\
We now use the fact that the edge $e_{3}$ is critical as well. Let $\varphi'$ be a proper $k$-edge-colouring of $G-e_{3}$ with $i \in \bar{\varphi}'(x_{3})$ and $j \in \bar{\varphi}'(y_{3})$. If $e_{1}$ and $e_{2}$ are coloured with the same colour, then $\varphi'_{j}\mid_{E(G_B)}$ is a proper $k$-edge-colouring of $G_B$ that is of type 1 or 2. On the other hand, if $\varphi'(e_{1}) \neq \varphi'(e_{2})$, then $\varphi'_{i}\mid_{E(G_A)}$ is a proper type 3, type 4 or type 5 $k$-edge-colouring of $G_A$.\\
In every case there are two proper $k$-edge-colourings, one of $G_A$ and one of $G_B$, that are of the same type. This contradicts the fact that $G$ is not $k$-edge-colourable.
\end{proof}
 
A graph without an even factor can be characterized as follows (see Theorem 6.2 (p. 221) in \cite{AK}). 

\begin{theo} \label{characterisation}
If $G$ is a graph, then $G$ has no even factor, if and only if there is an $X \subset V(G)$ with
\begin{align*} \label{eq1}
\tag{1}
\sum\limits_{v \in X} (d(v)-2) - q(G;X) < 0,
\end{align*}
where $q(G;X)$ denotes the number of components $D$ of $G-X$ such that $e_{G}(D,X) \equiv 1 \text{ (mod 2)}$.
\end{theo}

We will give a more detailed formulation of Theorem \ref{characterisation} with regard to a minimal set $X$ that satisfies inequality (\ref{eq1}).

\begin{theo} \label{reformulation}
If $G$ is a connected graph, then $G$ has no even factor, if and only if there is an $X \subset V(G)$ with the following properties: 	Let
$D_1, \dots, D_n$ be the components of $G-X$.

\begin{itemize}
\item[(a)] $\sum\limits_{v \in X} (d(v)-2) - q(G;X) < 0$,
\item[(b)] $e_{G}(D_{i},v) \leq 1$ for every $v \in X$ and every $i \in \{1,...,n\}$,
\item[(c)] $X$ is stable,
\item[(d)] $e_{G}(D_{i},X) \equiv 1 \text{ (mod 2)}$ for every $i \in \{1,...,n\}$,
\item[(e)]  $\sum\limits_{\substack{v\in X \\ d(v)\neq 2}} (d(v)-3) + \frac{1}{2} \sum\limits_{i=1}^{n} (e_{G}(D_{i},X)-3) < \vert \{ v\in X : d(v)=2\}\vert$.
\end{itemize}

\end{theo}

\begin{proof} By Theorem \ref{characterisation} it suffices to prove one
	direction. 
Let $G$ be a connected graph without an even factor. By Theorem \ref{characterisation}, there is a set that satisfies inequality (\ref{eq1}). Let $X \subset V(G)$ be the smallest set with $\sum_{v \in X}(d(v)-2)-q(G;X)<0$. We show that $X$ satisfies \textit{(b)} - \textit{(e)}.\\
For each $v \in X$ let $c(v)$ be the number of components $D$ of $G-X$ with $e_{G}(D,X) \equiv 1 \text{ (mod 2)}$ and $e_{G}(D,v) \geq 1$. We first prove $c(v)=d(v)$ for every $v \in X$, which implies properties \textit{(b)} -~\textit{(d)}, since $G$ is connected. \\
Let $x \in X$ and $X'=X\setminus \{x\}$. By the choice of $X$, the set $X'$ does not satisfy inequality \textit{(a)}. Furthermore, we observe that $-2 \vert X\vert + \sum_{v \in X} d(v) - q(G;X)$ is even. As a consequence,
\begin{align*}
0 & \leq \sum\limits_{v \in X'} (d(v)-2) - q(G;X')\\
& \leq -2\vert X \vert + 2 + \sum\limits_{v \in X} d(v) - d(x) - (q(G;X)-c(x))\\
& = -2\vert X \vert + \sum\limits_{v \in X} d(v) - q(G;X) + 2 - d(x) + c(x)\\
& \leq -2 +2 - d(x) + c(x).
\end{align*}
Thus, $d(x) \leq c(x)$, which implies $d(x) = c(x)$. Therefore, the set $X$ satisfies \textit{(b)} - \textit{(d)}.\\
Next, by using \textit{(c)} and \textit{(d)} we can transform \textit{(a)} to \textit{(e)} as follows:
\begin{align*}
&\sum\limits_{v \in X} (d(v)-2) - q(G;X) < 0\\
{\underset{\text{\textit{(d)}}}{\Leftrightarrow}} \quad &\sum\limits_{v \in X} \left(d(v)-2\right)  < n\\
\Leftrightarrow \quad &\frac{1}{2} \sum\limits_{v \in X} \left(d(v)-2\right)+\sum\limits_{v \in X} \left(d(v)-2\right) < \frac{3}{2} \sum\limits_{i=1}^{n} 1\\
{\underset{\text{\textit{(c)}}}{\Leftrightarrow}} \quad &\frac{1}{2}\sum\limits_{i=1}^{n} \left(e_{G}(D_{i},X)\right)- \vert X \vert +\sum\limits_{v \in X} \left(d(v)-2\right) <  \sum\limits_{i=1}^{n} \frac{3}{2}\\
\Leftrightarrow \quad &\frac{1}{2}\sum\limits_{i=1}^{n} \left( e_{G}(D_{i},X)-3\right) +\sum\limits_{v \in X} \left(d(v)-3\right) < 0\\
\Leftrightarrow \quad &\frac{1}{2}\sum\limits_{i=1}^{n} \left( e_{G}(D_{i},X)-3\right) +\sum\limits_{\substack{v\in X \\ d(v)\neq 2}} \left(d(v)-3\right) - \vert\{v\in X: d(v)= 2\} \vert < 0\\
\Leftrightarrow \quad &\sum\limits_{\substack{v\in X \\ d(v)\neq 2}} \left(d(v)-3\right) + \frac{1}{2} \sum\limits_{i=1}^{n} \left(e_{G}(D_{i},X)-3\right) < \vert \{ v\in X: d(v)=2\}\vert.
\end{align*}
\end{proof}

\section{Proof of Theorem \ref{main result}}

For $k=3$ there is nothing to prove. Let $ k > 3$. 
Let $G$ be a $k$-critical graph without an even factor. Hence, there is a subset $X \subset V(G)$ that satisfies conditions \textit{(a)} -\textit{(e)} of Theorem \ref{reformulation}. We show that $X$ contains more than $2k-6$ divalent vertices.\\
Let $D_{1},...,D_{n}$ be the components of $G-X$ and $g: \{D_{1},...,D_{n}\} \to \mathbb{R}$ with
\begin{center}
$g(D_{i}):=\sum\limits_{v \in N(D_{i})} \frac{d(v)-2}{d(v)}$ for $i \in \{1,...,n\}$.
\end{center}
Properties \textit{(a)} -\textit{(d)} imply
\begin{align*}
\sum\limits_{i=1}^{n} g(D_{i})	\quad =	\quad \sum\limits_{i=1}^{n} \sum\limits_{v \in N(D_{i})} \frac{d(v)-2}{d(v)} 	\quad{\underset{\text{\textit{(b)},\textit{(c)}}}{=}}	\quad \sum\limits_{v \in X} d(v)-2 	\quad{\underset{\text{\textit{(a)}}}{<}}	\quad q(G;X) 	\quad{\underset{\text{\textit{(d)}}}{=}}	\quad n.
\end{align*}
Thus, there is at least one component $D\in \{D_{1},...,D_{n}\}$ with $g(D) < 1$. 
Every critical graph does not contain a vertex of degree 1. Therefore, 
there are at most two vertices in $N(D)$ that are not divalent. Moreover, if $N(D)$ contains two vertices of degree at least 3, then one of them is of degree 3 and the other is of degree at most 5. Furthermore, since every critical graph is bridgeless, the component $D$ has at least three neighbours in $X$ by \textit{(b)} and \textit{(d)}. In conclusion, we can assume $N(D)=\{x,y,w_{1},...,w_{l}\}$, where $l \geq 1$, the vertices $w_{1},...,w_{l}$ are divalent and either $d(y)=2$, or $d(y)=3$ and $d(x)\leq5$. We consider the following two cases:\\
\ \\
\underline{Case 1: $d(x)<k$}\\
By condition \textit{(b)} and Lemma \ref{cutlemma} it follows that
\begin{align*}
l > k(k-d(y))-d(x)+1 \geq k(k-3)-k+2 = k(k-4)+2 \geq 2k-6.
\end{align*}
\ \\
\underline{Case 2: $d(x)=k$}\\
If there are three components adjacent to $x$ such that each has exactly three edges to $X$, then none of this components is adjacent with a divalent vertex by Lemma \ref{3-cutlemma}. In conclusion, we obtain with property \textit{(b)}
\begin{align*}
\sum_{\substack{i \in \{1,...,n\} \\ x \in N(D_{i})}} g(D_{i}) \geq d(x)-2+6\left(\frac{1}{3}\right)=d(x).
\end{align*}
Since $\sum_{i=1}^{n} g(D_{i})<n$, there is another component $D' \in \{D_{1},...,D_{n}\}$ with $g(D')<1$, but $x \notin N(D')$. If $D'$ is not adjacent to a vertex of degree $k$, then $N(D')$ contains at least $2k-6$ divalent vertices since Case 1 applies. Otherwise $X$ contains at least two vertices of degree $k$. Since $G$ is bridgeless, property \textit{(d)} implies that $e_{G}(D_{i},X) \geq 3$ for every $i \in \{1,...,n\}$. In conclusion, property \textit{(e)} implies that
\begin{align*}
\vert \{ v\in X: d(v)=2\}\vert > 2(k-3)=2k-6.
\end{align*}
If at most two components adjacent to $x$ have exactly three edges to $X$, then by properties \textit{(b)} - \textit{(d)} there are at least $d(x)-2$ components such
that each has at least five edges to $X$. Therefore, property \textit{(e)} implies that
\begin{align*}
\vert \{ v\in X: d(v)=2\}\vert>d(x)-3+(d(x)-2)\frac{1}{2}(5-3)=2k-5,
\end{align*}
and the proof is completed.

\subsection*{Acknowledgements:}
We thank the anonymous reviewers for careful reading and their helpful comments, which in particular led to shorter proofs of Lemma \ref{3-cutlemma} and Theorem \ref{reformulation}.

\end{document}